\documentclass[fleqn,leqno]{article}[10pt]%
\usepackage{amsmath}
\usepackage{amsfonts}
\usepackage{amssymb}
\usepackage{graphicx}%

\usepackage[a4paper,hypertex]{hyperref}

\newtheorem{theorem}{Theorem}

\newtheorem{corollary}[theorem]{Corollary}

\newtheorem{definition}[theorem]{Definition}

\newtheorem{lemma}[theorem]{Lemma}

\newtheorem{remark}[theorem]{Remark}

\newenvironment{proof}[1][Proof]{\noindent\textbf{#1.} }{\ \rule{0.5em}{0.5em}}

\def\bbb#1{\hbox {{\gordas #1}}}

\font\gordas = msbm10 at 12pt
\def\bbb#1{\hbox {{\gordas #1}}}

\def\zet{{\bbb Z}}

\def\zet{{\bbb Z}}

\def\cal{\mathcal}

\begin{document}

\title{  Transience and recurrence   of random walks \\
on percolation clusters in an ultrametric space}

\author{\bf D. A. Dawson $^{1}$ L.G. Gorostiza $^{2}$}

\date{\small  \today}

\maketitle

\begin{abstract}
We study transience and recurrence of simple random walks on percolation clusters in the
hierarchical group of order $N$, which is an ultrametric space. The connection probability on the
hierarchical group for two points separated by distance $k$ is of the form $c_k/N^{k(1+\delta)},
\delta>0$, with $c_k=C_0+C_1\log k+C_2k^\alpha$, non-negative constants $C_0, C_1, C_2$, and
$\alpha>0$. Percolation occurs for $\delta<1$, and for the critical case, $\delta=1$ $\alpha>0$ and
sufficiently large $C_2$. We show that in the case $\delta<1$ the walk is transient, and in the
case $\delta=1,C_2>0,\alpha
>0$  there exists a critical $\alpha_c\in(0,\infty)$ such that the walk is recurrent for
$\alpha<\alpha_c$ and transient for $\alpha>\alpha_c$. The proofs involve ultrametric random
graphs,  graph diameters, path lengths, and electric circuit theory. Some comparisons are made with
behaviours of simple random walks on long-range percolation clusters in the one-dimensional
Euclidean lattice.
\end{abstract}

\vspace*{0.5cm}

\medskip

  \vspace*{0.5cm}

\noindent {\footnotesize $^{1)}$ Carleton University, Ottawa, Canada\\
e-mail: ddawson@math.carleton.ca, supported by an NSERC Discovery grant.}

\smallskip

\noindent {\footnotesize $^{2)}$ {CINVESTAV, Mexico City, Mexico\\
e-mail: lgorosti@math.cinvestav.mx}, supported by CONACyT grant 98998.}

\newpage
\tableofcontents
\newpage

\section{Introduction}
\setcounter{secnumdepth}{4} \setcounter{equation}{0}
\numberwithin{equation}{section} \numberwithin{theorem}{section}

 Network science is an active field of
research due to its many areas of application (statistical physics, biology, computer science,
communications, economics, social sciences, etc.), and to the interesting mathematical problems
that it gives rise to, many of which remain open. Percolation plays an important role, for example
in the study of robustness of networks. Hierarchical networks occur in models where there is a
multiscale organization with an ultrametric structure, e.g., in statistical physics (in particular
disordered spin systems), protein dynamics,  population genetics and computer science. Several
areas of physics where ultrametric structures are present were overviewed in \cite{RTV}. An
ultrametric model in population genetics was introduced in \cite{SF}. Hierarchical organizations in
complex networks were discussed in \cite{BA}.
 A classical ultrametric space is the set of $p$-adic numbers.
A review of many areas where $p$-adic analysis is used, specially in physics including quantum
physics, appeared in \cite{DKKV}.  The ultrametric space we deal with in this paper is $\Omega_N$,
the {\it hierarchical group of order } $N$, described at the end of the Introduction. Background on
ultrametric spaces can be found e.g. in \cite{Sc}.

Stochastic models on hierarchical groups have played a fundamental role in mathematical physics and
population biology. Dyson \cite{D} introduced such a structure in order to gain insight on the
study of ferromagnetic models on the Euclidean lattice of dimension 4, as it provides a
``caricature'' of the Euclidean lattice in dimensions ``infinitesimally close'' to 4. A reason for
this approach is that it is possible to carry out a renormalization group analysis in a rigorous
way in hierarchical groups \cite{BS,CE}. Hierarchical groups have also been used in the study of
self-avoiding random walks in four dimensions \cite{BEI}, Anderson localization in disordered media
\cite{BHS, KM}, mutually catalytic branching in population models \cite{CDG,DGZ}, interacting
diffusions \cite{DG1, DG2}, occupation times of branching systems \cite{DGW1, DGW2}, search
algorithms \cite{K1,K2}. Thus, stochastic models, in particular random walks, on ultrametric spaces
are a natural field of study. A class of random walks on hierarchical groups, called {\it
hierarchical random walks}, and their degrees of transience and recurrence  were studied in
\cite{DGW3,DGW4} (and references therein). A related model in the context of spin-glass was treated
in \cite{O}. Other properties of systems of hierarchical random walks appeared in \cite{BGT1,
BGT2}. An analogous class of hierarchical random walks on the $p$-adic numbers was studied in
\cite{AKZ,AZ}. L\'evy processes on totally disconnected groups (including the $p$-adic integers)
were discussed in \cite{E}, pseudodifferential equations and Markov processes over $p$-adics were
treated in \cite{CZ}. A random walk model for the dynamics of proteins was discussed in \cite{ABZ}.
In this case the states of the walk are related to the local minima of the potential energy of a
protein molecule. These are a few representative references on stochastic models on ultrametric
spaces.

With these precedents and previous work on percolation in hierarchical groups \cite{DGo1,DGo2,KMT},
we were motivated  to investigate the behaviour of random walks on percolation clusters in those
groups, and to compare results with similar ones for random walks on long-range percolation
clusters in  Euclidean lattices (referred to below).

The renormalization method for the study of percolation in hierarchical networks involves
ultrametric random graphs. An {\it ultrametric random graph} $U\!RG(M,d)$ is a  graph on a finite
set of  $M$ elements with an ultrametric $d$ and connection probabilities $p_{x,y}$ that are random
and depend on the distance $d(x,y)$ (see \cite{DGo2}, Section 3.4). A more detailed description
related to the model is given in Section 2.

In \cite{DGo1} we studied asymptotic percolation in $\Omega_N$ in the limit $N\to\infty$ (mean
field percolation) with a certain class of connection probabilities depending on the   distance
between points. In this case it was possible to obtain a necessary and sufficient condition for
percolation. The Erd\H os-R\'enyi theory of giant components of random graphs was a useful tool,
although there are significant differences between classical random graphs and ultrametric ones. Percolation in $\Omega_N$  with fixed $N$ is technically more
involved, and so far only sufficient conditions for percolation or for its absence are known. This
was studied in \cite{DGo2}, where connectivity results of Erd\H os-R\'enyi graphs played a basic
role. At the same time an analogous model was studied in \cite{KMT} using different methods. In
\cite{DGo2} the ``critical case'' was analyzed in more depth.
 A relationship between the results of \cite{DGo2} and \cite{KMT} was given in \cite{DGo2} (Remark 3.2).  The relevance of percolation in hierarchical groups has been noted  for contact processes \cite{AS} and
 epidemiology \cite{G}.

In \cite{DGo2} we studied percolation in the hierarchical group   $(\Omega_N,d)$, integer $N\geq
2$, ultrametric $d$, with
 probability of connection between two points ${\bf x}$ and
${\bf y}$ such that $d({\bf x},{\bf y})=k\geq 1$  of the form $p_{{\bf x},{\bf y}}=
c_k/N^{(1+\delta)k}$, where $\delta>-1$ and the $c_k$ are positive constants, all connections being
independent. Here we restrict to $\delta >0$. The results  refer to existence of {\it percolation
clusters} (infinite connected sets) of positive density. {\it Percolation} is said to occur if a
given point of $\Omega_N$ belongs to a percolation cluster with positive probability. The specific
point does not matter because the model is translation-invariant. By ultrametricity, percolation is
possible only if there exists arbitrarily large $k$ such that $c_k>0$ (otherwise all connected
components are finite). Thus, percolation in $\Omega_N$ can be regarded as long-range percolation.
Briefly, in \cite{DGo2} the results are: if $\delta<1$ and $c=\inf_kc_k$ is large enough, then
percolation occurs, if $\delta>1$ and $\sup_kc_k<\infty$, then percolation does not occur, and for
the critical case, $\delta=1$, which is the most delicate, percolation may or may not occur
according to some special forms of $c_k$ such that $c_k\to\infty$ as $k\to\infty$. When percolation
occurs the infinite cluster is unique.

In the critical case $c_k$ was taken of the form
\begin{equation}
\label{eq:1.1} c_k=C_0+C_1\log k+C_2k^\alpha
\end{equation}
with non-negative constants $C_0,C_1,C_2$, and $\alpha>0$, and for the case $C_2>0$ percolation was
established for $\alpha>2$ and any $C_1$ if $C_0$ and $C_2$ are large enough \cite{DGo2} (Theorem
3.3(a)). The proof was based on a renormalization argument of the type used in statistical physics.
Results for the  case $C_2=0$ were also obtained, in particular if $C_1<N$, then percolation does
not occur for any $C_0$. Here we show that percolation occurs for any $\alpha>0$ if $C_2$ is large
enough (Theorem \ref{T.1}), which  was an open problem in \cite{DGo2} (section (3.4)). Here we need
$\alpha >0$ for the results on random walks. The proof uses the renormalization ideas introduced in
\cite{DGo2}, but in a different way which is more intrinsic to the model.

The renormalization approach in \cite{DGo2} was applied for a  preliminary percolation result
replacing $c_k$ with $c_k^\prime$ of the form
\begin{equation}
\label{eq:1.2} c^\prime_{k_n}=C+a\log n\cdot n^{b\log N},
\end{equation}
constants $C\geq 0, a>0, b>0$, where
\begin{equation}
\label{eq:1.3} k_n=\lfloor Kn\log n\rfloor, n=1,2,\ldots,
\end{equation}
 constant
$K>0$, and $c^\prime_{k_n}\leq c_k^\prime\leq c^\prime_{k_{n+1}}$ for $k_n<k<k_{n+1}$, with some
technical conditions on $K$ and $b$ \cite{DGo2} (Theorem 3.5(b)), which was used  to prove
percolation for $\alpha>2$. The conditions were $2/\log N<K<b$ (with a minor modification it is
possible to have also $K=b$).
 The proof of percolation for $\alpha>2$ in \cite{DGo2} is based on the relationship between (1.1)
  and (1.2) with  $\alpha >b\log N$ (proof of Theorem 3.3 in \cite{DGo2}).
  Note that $c_k^\prime\leq c_k$.  We will write $c_k$ for $c_k^\prime$ in (\ref{eq:1.2}) for
  simplicity of notation, and no confusion should arise. All one needs to remember regarding (\ref{eq:1.1}) and (\ref{eq:1.2}) is
  $\alpha
 >b\log N$. Percolation with $c_k^\prime$ imples percolation with $c_k$ as in the proof of Theorem
  3.3 in \cite{DGo2}.
The scheme with  (1.2), (1.3) is not  used for the proof of percolation here, but it constitutes a
technical tool for the study of behaviour of random walks on percolation clusters regarding some
properties of the  clusters, hence we will need to refer to some techniques in \cite{DGo2}.

The  main results in the paper refer to transience and recurrence of simple (nearest neighbour)
random walks on the percolation clusters in $\Omega_N$. We show that the random walk is transient
for $\delta<1$ (Theorem \ref{T.trans1}), and in the critical case, $\delta=1$, $C_2>0$, there
exists a critical  $\alpha_c\in(0,\infty)$ such that the random walk is recurrent for $\alpha
<\alpha_c$
 and transient for $\alpha >\alpha_c$ (Theorem \ref{T.main2}).

These
 results are comparable in part with those on long-range percolation in the one-dimensional Euclidean lattice $\zet$ with connection probabilities of the
form $\beta|x-y|^{-s}$ as $|x-y|\to \infty$, although the Euclidean and the ultrametric structures
are quite different. Long-range percolation in $\zet$ with those connection probabilities was
introduced by Schulman \cite{S}, and studied further for $\zet^d$ by Newman and Schulman \cite{NS},
and Aizenman and Newman \cite{AN}. Berger \cite{B} studied transience and recurrence of random
walks on the percolation clusters in $\zet^d$ for $d=1,2$. The results for $d=1$ are, roughly, that
percolation can occur if $1<s\leq 2$, and does not occur if $s>2$, and if $1<s<2$, then the  walk
is transient, and if $s=2$, then the walk is recurrent. Hence the results agree for $\Omega_N,
0<\delta<1$, and  $\zet, 1<s<2$,  by using the ultrametric $\rho({\mathbf x}, {\mathbf
y})=N^{d({\mathbf x}, {\mathbf y})}$ (``Euclidean radial distance'')
 on $\Omega_N$, and
$s=\delta+1$. But there is a significant difference. Percolation in $\zet$ can be obtained  by
increasing the probability of connection between nearest neighbors ( separated by distance $1$)
\cite{B} (Theorem 1.2), whereas for $\Omega_N$
 short-range connections  play no role due to ultrametricity.
Our results for $\delta=1$ with $c_k$ given by (\ref{eq:1.1}) would correspond to the case on
$\mathbb{Z}$  with  $s=2$ taking $\beta$ to be  a function of distance. Heat kernel bounds and
scaling limits for the walks on the long-range percolation clusters of $\mathbb{Z}^d$, $d\geq 1$,
were obtained in \cite{CS1,CS2}. It would be interesting to find similar results for the
hierarchical group.

Grimmett et al \cite{GKZ} studied the behaviour of random walks on (bond) percolation clusters in
the Euclidean lattice $\zet^d$ using electric circuit theory \cite{DS} (see also  \cite{BPP,P}).
Recurrence of the walk for $d\leq 2$ follows  directly from recurrence on the whole space $\zet^d$,
whereas transience for $d\geq 3$ was difficult to prove. For recurrence in long-range percolation
the situation is  different because the walk is only defined on the percolation cluster (both in
$\zet^d$ and in $\Omega_N$). Although the models on $\zet^d$ and $\Omega_N$ are quite different, we
are able to use some of the basic ideas on the relationship between reversible Markov chains and
electric circuits (see e.g. \cite{DS,LP,KL}) that have been used for $\zet^d$, but in the case of
$\Omega_N$ the ultrametric structure plays a fundamental role.

The transience and recurrence behaviours of  walks on the percolation clusters are determined
basically by the ultrametric geometry of $(\Omega_N,d)$ and the form of the connection
probabilities, rather than by detailed properties of the structures of the percolation clusters.
The proofs involve some properties of the  clusters, in particular cutsets, graph diameters and
lengths of paths.

We end the Introduction by recalling  ($\Omega_N,d$) and some things about it. For an integer
$N\geq 2$, the {\it hierarchical group ({\rm also called} hierarchical lattice) of order} $N$ is
defined as \vglue.3cm \noindent \centerline{ $\Omega_N=\{{\bf x}=(x_1,x_2,\ldots): x_i \in\zet_N,
x_i=0\quad{\it a.a.i}\}.$ } \vglue.3cm \noindent with addition componentwise mod $N$, where
$\zet_N$ is the cyclic group of order $N$.  The {\it hierarchical distance} on $\Omega_N$, defined
as
$$
d({\bf x},{\bf y})= \left\{
\begin{array}{lll}
0&{\rm if}&{\bf x}={\bf y}, \\
\max \{i:x_i\neq y_i\}&{\rm if}&{\bf x}\neq{\bf y},
\end{array}\right.
$$
satisfies the strong (non-Archimedean) triangle inequality,
$$d({\bf  x},{\bf y})\leq \max \{d({\bf x},{\bf z}),
d({\bf z},{\bf y})\}\quad\hbox{\rm for any}\quad {\bf x},{\bf y},{\bf z}.$$ Hence $(\Omega_N,d)$ is
an ultrametric space, and it can be represented as the top of an infinite regular $N$-ary tree
where the distance between two points is the number of levels from the top to their closest common
node. The point $(0,0,\ldots)\in\Omega_N$ is taken as origin and denoted by ${\bf 0}$. The
probability of an edge connecting two points ${\bf x}$ and ${\bf y}$ is given by
\begin{equation}
\label{eq:1.4} p_{{\bf x},{\bf y}}=\min \left(\frac{c_k}{N^{(1+\delta)k}},1\right)\,\, {\rm if}\;\;
d({\bf x},{\bf y})=k,
\end{equation}
where $\delta>0$ and $c_k>0$ for every $k$, all edges being independent.
 Note that any point in a percolation cluster has a finite (random) number of
neighbours since the number has finite expectation, hence the simple (nearest neighbour) random
walk on a percolation cluster is well defined.

An essential property of ultrametric spaces that differentiates them from Euclidean spaces is that
two balls are either disjoint or one is contained in the other. The following definitions and
properties are used throughout. The ball of diameter $k\geq0$ containing ${\bf x}\in\Omega_N$ is
defined as $B_k({\bf x})=\{{\bf y}:d({\bf x},{\bf y})\leq k\}$. Those balls are generally referred
to as $k$-{\it balls}. They contain $N^k$ points. For $k>0$, a $k$-ball is the union of $N$
disjoint $(k-1)$-balls that are at distance $k$ from each other. For $j>k>0$, we call $B_j({\bf
0})\backslash B_k({\bf 0})$ the {\it annulus $(k,j]$, {\rm or} $(k,j]$-annulus}. It contains
$N^j(1-N^{k-j})$ points. A $j$-ball is the union of $N^{j-k}$ disjoint $k$-balls. The $k$-balls in
the $(k,j]$-annulus are at distance at least $k+1$ and at most $j$ from each other. This and (1.4)
allow to obtain upper and lower bounds for the probability that subsets of two $k$-balls in the
$(k,j]$-annulus are connected by at least one edge. Such bounds are used in the proofs.

In Section 2 we prove percolation for $\delta =1,\;\alpha >0$. In Section 3 we prepare the tools
for the proofs of transience and recurrence on the random walks based on the properties of the
clusters. In Section 4  we give the results and proofs of transience and recurrence of the walks
using electric circuit theory.

\section{Percolation in $\Omega_N$ for $\delta=1$}\label{S.perc}

The results for $\delta<1$ and $\delta >1$ have been mentioned in the Introduction. For $\delta=1$
we regard the model as
 the  infinite random graph
\begin{equation*}\mathcal{G}^\infty_N=\mathcal{G}^\infty_N(C_0,C_1,C_2,\alpha):=\mathcal{G}(V_\infty,\mathcal{E}_\infty)\end{equation*}
with vertices $V_\infty=\Omega_N$  and edges $\mathcal{E}_\infty$, and with the probability of
connection by an edge $({\bf x},{\bf y})$ \begin{equation}\label{E.2.1} P(({\bf x},{\bf
y})\in\mathcal{E}_\infty)=p_{{\bf x},{\bf y}}=\min \left(\frac{c_k}{N^{2k}},1\right)\,\, {\rm
if}\;\; d({\bf x},{\bf y})=k,\end{equation} all connections being independent, and the $ c_k$ are
of the form \begin{equation}\label{E.2.2} c_k=C_0+C_1\log k+ C_2k^\alpha,\end{equation} with
constants $C_2
> 0$ and $\alpha > 0$.  For simplicity of notation we set, without loss of generality, $C_0=C_1=0$
in (\ref{E.2.2}).  The graph $\mathcal{G}^\infty_N$ is a limit of finite graphs of diameter $k$
which we will study as $k\to\infty$.

\begin{theorem}\label{T.1} For sufficiently large $C_2$, there exists a unique percolation cluster  of positive density at least
 $\varepsilon=\varepsilon(N,C_2,\alpha)$  in
$\Omega_N$ to which $\mathbf{0}$ belongs with positive probability.

\end{theorem}

\medskip

The proof of this result is given in the next subsection. We begin with the formulation for the
renormalization method.

\subsection{The hierarchy of random graphs}

A collection of vertices in a subset of $\Omega_N$ any two of which are linked by a path of edges
is called a {\em cluster} of the subset. We consider for each $k$-ball a maximal cluster with edges
only within the ball and not through paths going outside the ball (i.e., all edges of length $\leq
k$ not in the cluster are deleted). If there are more than one (maximal) cluster, then one of them
is chosen uniformly at random. In this way each $k$-ball has a unique attached cluster. The proof
will be based on the connections between the clusters in $k$-balls. When we refer to connections
between $k$-balls we mean direct edge connections (one or more) between their clusters. The {\em
density} of a $k$-ball is the size of its cluster normalized by the size of the ball $(N^k)$. Due
to our assumptions, the densities of different $k$-balls are i.i.d. An infinite connected subset of
$\Omega_N$ is called a {\em percolation cluster}.

 The main idea is to consider the distribution of the  clusters in the balls $B_k({\bf 0})$ of increasing $k$ by relating the random graphs in these balls to a
hierarchy of ultrametric random  graphs.

\subsubsection{Erd\H{o}s-R\'enyi graphs with random  weights, ultrametric random graphs}

In the classical  graph $\mathcal{G}(n,p)$ introduced by Gilbert \cite{Gil}  these graphs have a
set of $n$ vertices denoted by $V$ and there is an edge between each pair of vertices with
probability $p$ with these assigned independently for different pairs (see e.g.  \cite{Bo,JLR} for
background).  The behaviour of these graphs together with the random graphs $\mathcal{G}(n,m)$ in
the limit as $n\to\infty$ were studied by Erd\H{o}s and R\'enyi in a series of important papers. We
consider a modification of those graphs, namely, $\mathcal{G}(N,\{x_i\}_{i\in V})=
\mathcal{G}(N,\{p(x_i,x_j)\})$ in which the vertices have independent random weights $\{x_i\}_{i\in
V}$, and the probability that $i$ and $j$ are connected by an edge is a function $p(x_i,x_j)$, the
edges chosen independently conditioned on the weights. These are ultrametric random graphs as
stated in the Introduction.
\bigskip

 Given $\mathcal{G}^\infty_N$ as above we now introduce a sequence of related finite ultrametric random
 graphs\\
$\mathcal{G}_k(N,\{X_{k-1}(i)\}_{i\in V_k}),\,k\geq 1$, where $X_0(i)=1$ and for $k\geq 2$, the
$\{X_{k-1}(i)\}_{i\in V_k}$ are the densities of  the $N$ disjoint $(k-1)$-balls in
$B_k(\mathbf{0})$ indexed by $V_k,\, |V_k|=N$. The densities $\{X_{k-1}(i)\}_{i\in V_{k}}$ are
i.i.d. $[0,1]$-valued random variables for each
  $k$,
   \begin{equation}\label{E.2.3} X_{k-1}(i) =\frac{|\mathcal{C}_{k-1}(i)|}{N^{k-1}},\quad  i\in
   V_{k},\end{equation}
   where $\mathcal{C}_{k-1}(i)$ denotes the cluster in the $i$th $(k-1)$-ball.
    For $N$  fixed
   our aim is to  determine what happens as $k\to\infty$. Properties of the graph
   $\mathcal{G}_k(N,\{X_{k-1}(i)\}_{i\in V_k})$ as $k\to\infty$ provide information on $\mathcal{G}^\infty_N$, hence the behaviour of the cluster of
    $B_k(\mathbf{0})$ as $k\to\infty$ will imply a result on
   percolation in $\Omega_N$.

    We denote the distribution of $X_k(i)$ by $\mu_k\in\mathcal{P}([0,1])$.  Then we have for each $k$,
\[ \mu_k=\Phi_k(\mu_{k-1}),\]  where  $\Phi_k$ is a {\em renormalization mapping}  \[\Phi_k:\mathcal{P}([0,1])\to\mathcal{P}([0,1]).\]

 \medskip

Note that $\mu_k$ depends on the edges within a $(k-1)$-ball (which determine $\mu_{k-1}$) and also
on the edges between different $(k-1)$-balls in a $k$-ball, and that $\mu_k$ is an atomic measure.

We now analyse the sequence $\mu_k,\,k\ge1 , $ for the class of connection probabilities
(\ref{E.2.1}), (\ref{E.2.2}).

  In order to prove percolation of positive density it suffices to construct the sequence of $[0,1]$-valued random
variables $X_k,\, k\geq 1,$ satisfying \[ P(X_k> a)= \mu_k((a,1])\text{  for all }a\in(0,1),\] and
existence of $a>0$ such that
\[ \liminf_{k\to\infty} P(X_k>a)>0.\]

Consider two $(k-1)$-balls in a $k$-ball (which are at distance $k$ from each other) having
densities $x_1,x_2$ respectively, and define
\begin{equation}
p(x_{1},x_{2},k)    =P(\text{two }(k-1)\text{-balls in a }k\text{-ball  with densities
}x_1,x_2\text{ are connected}).\end{equation} Note that  $ p(x_{1},x_{2},k)$ is an increasing
function of $x_1$ and $x_2$. Then from (\ref{E.2.1}), (\ref{E.2.2}),
\begin{equation}\label{E.2.4}
p(x_{1},x_{2},k) =1-\left(  1-\frac{C_{2}k^{\alpha}}{N^{2k}}\right)  ^{N^{2(k-1)}%
x_{1}x_{2}},
\end{equation}
and  \begin{equation}\label{E.2.5}  p(x_1,x_2,k)\sim1-e^{-\left( {C_{2}k^{\alpha}}/{N^2}\right)
x_{1}x_{2}}\text{  for large }k.
\end{equation}

\bigskip

\subsubsection{Proof of Theorem \ref{T.1} }
  The idea of the proof is to obtain lower bounds on the expected values of the sequence of random
variables $\{X_k(\mathbf{0})\}_{k\geq 1}$,
 where $X_k(\mathbf{x})$
denotes the {\em density} of  $B_k(\mathbf{x})$, that is $X_k(\mathbf{x})=
{|\mathcal{C}_k(\mathbf{x})|}/{N^k}$, where $\mathcal{C}_k(\mathbf{x})$ denotes the cluster in
$B_k(\mathbf{x})$. Then as remarked above $\mu_k$ is the probability law of $X_k(\mathbf{x})$ which
is independent of $\mathbf{x}$ and the latter will  be suppressed.  By our assumptions
$\{X_k(\mathbf{x}_i)\}$ are independent if for $i\ne j$, $d(\mathbf{x}_i,\mathbf{x}_j)\geq k+1$.

\begin{lemma}\label{L.2} Let $X$ be a random variable with values in $[0,1]$ and $0<a<1$. Then
\begin{equation}\label{E.2.6}
P(X\geq a/2)\geq \frac{E[X]-a/2}{1-a/2}.
\end{equation}

\end{lemma}

\begin{proof}
\bigskip Let $p=P(X\geq a/2)$. Then
\begin{align}
&  E[X]\leq p+\frac{a}{2}(1-p),\nonumber\end{align} hence
\begin{equation*}
p  \geq \frac{E[X]-a/2}{1-a/2}.
\end{equation*}
\end{proof}

\begin{lemma}\label{L.3} Assume that%
\begin{equation}\label{E.2.7}
\liminf_{k\rightarrow\infty}E[X_{k}]=\liminf_{k\rightarrow\infty
}E[\frac{|\mathcal{C}_{k}|}{N^{k}}]=a>0,
\end{equation} for some $a>0$.
Then percolation of positive density occurs.
\end{lemma}
\begin{proof}
\noindent Lemma \ref{L.2} implies
\begin{equation*}
\liminf_{k\rightarrow\infty}P(\frac{|\mathcal{C}_{k}|}{N^{k}}\geq a/2)\geq\frac{a}{2-a}.
\end{equation*}

Assume that the expected value of the density of $B_k(\mathbf{0})$ is at least $a/2$ for some $a>0$
and all large $k$. Then  by transitivity (cf. \cite{DGo2}, Lemma 5.6)
\begin{equation}\label{E.2.8}
\liminf_{k\rightarrow\infty}P({\mathbf{0}}\text{ belongs to the cluster of }%
B_k(\mathbf{0}))\geq\frac{a^{2}}{2(2-a)}>0.
\end{equation}

On the other hand if we assume that the density of $B_k(0)$ tends to $0$ w.p.1, then (\ref{E.2.7})
does not hold.

\end{proof}
\medskip

To prove Theorem \ref{T.1} we will show that the hypothesis of Lemma \ref{L.3} is satisfied.

We write the sequence of weights $ X_k,\,k\geq 1$, as follows.
\begin{equation}\label{E.2.9}
 X_{k+1}=\frac{1}{N}\sum_{i=1}^N 1_{i\in \mathcal{C}^*_{k+1}} X_{k,i},
\end{equation}
where  $\{ X_{k,i},i=1,\dots,N\}$ denote the (i.i.d.) densities of the $N$ disjoint $k$-balls in
$B_{k+1}(\mathbf{0})$,
 and $\mathcal{C}^*_{k+1}\subset V_{k+1}$  is the set of
indices of the underlying $k$-balls (some $k$-balls may have null weights). Hence
\begin{equation}\label{E.2.10}
E[X_{k+1}]\leq E[X_k]\;\; \text{for all }k.
\end{equation}

\medskip

 Now consider the random graph $\mathcal{G}_k(N,\{x_1,\dots,x_N\})= \mathcal{G}_k(N,\{p(x_{i},x_{j},k)\})$, where $p(x_i,x_j,k)$ is given by (\ref{E.2.4}).
 Given the densities $( X_{k-1,1},\dots, X_{k-1,N})$ $=(x_{1},
\dots,x_{N})$, then the probability that all $N$ $(k-1)$-balls in a $k$-ball are connected is
\begin{equation}\label{E.2.11}
 P(X_{k}   =\frac{1}{N}(x_{1}+\dots +x_{N})\text{ }|(x_{1},\dots,x_{N})) = P(
\mathcal{G}_k(N,\{p(x_{i},x_{j},k)\})\text{  is connected}).\end{equation}

\noindent If $x_i\geq\varepsilon >0$ for $i=1,\dots,N$, then
\begin{equation}\label{E.2.12} P( X_{k}   =\frac{1}{N}(x_{1}+\dots +x_{N})\text{ }|(x_{1},\dots,x_{N}))
\geq P( \mathcal{G}_k(N,p(\varepsilon,\varepsilon,k)) \text{ is connected}).\end{equation} If all
the $(k-1)$-balls are isolated, then
\begin{equation}\label{E.2.13} P( X_{k}   =\frac{1}{N}(x_{1}\vee\dots\vee
x_{N})|(x_{1},\dots,x_{N}))=\prod_{i<j=1}^N (1-p(x_{i},x_{j},k)),
\end{equation}
where the right side is the probability that no pair of $k$-balls is connected.


By the independence of the densities of different$(k-1)$-balls,
\begin{equation}\label{E.2.14}
P((X_{k-1,1},\dots,X_{k-1,N})=(x_{1},\dots,x_{N}))=\prod_{i=1}^N \mu_{k-1}(x_{i}).
\end{equation}

\medskip

We now can state a stronger form of (\ref{E.2.7}).

\begin{lemma}\label{L.4} For sufficiently large $C_2$ there exists $a>0$ such that as
$n\to\infty$,

\begin{equation}\label{E.2.15} E[X_n]\to a,\end{equation}
\begin{equation}\label{E.2.16} Var[X_n]\to 0,\end{equation}
and
\begin{equation}\label{E.2.17}
\mu_n\Rightarrow \delta_a.
\end{equation}
\end{lemma}

We first consider the case $N=2$ to illustrate the idea of the proofs of Lemma \ref{L.4} and
Theorem \ref{T.1}.
\bigskip

\noindent\textbf{Proof of Lemma \ref{L.4} for $N=2$.} Fix $0<\varepsilon<1$ (to be chosen
sufficiently small), and let
\begin{eqnarray}\label{E.2.18}&& q_n(\varepsilon)= \sup\{1-p(x_1,x_2,n):x_1,x_2\geq
\varepsilon\}\\&&=P(\text{two } (n-1)\text{-balls in an }n\text{-ball with densities } \geq
\varepsilon\text{ are not connected})\nonumber
\end{eqnarray}
Then $q_n(\varepsilon)$ is decreasing in $n$ for large $n$, and from (\ref{E.2.4}), (\ref{E.2.5}),
\begin{equation}\label{E.2.19}
1-q_n(\varepsilon)  \geq 1-\left(  1-\frac{C_{2}n^{\alpha}}{N^{2n}}\right)  ^{N^{2(n-1)}%
x_{1}x_{2}}\sim1-e^{-\left(  {C_{2}n^{\alpha}}/{N^2}\right) x_{1}x_{2}}\quad\text{for large }n,
\end{equation}
\begin{equation}\label{E.2.20}
q_n(\varepsilon)\leq (q(\varepsilon))^{n^\alpha},\end{equation} where
\begin{equation}\label{E.2.21}  q(\varepsilon):= e^{-C_2\varepsilon^2/N^2},
\end{equation}
hence \begin{equation}\label{E.2.22} \sum_n q_n(\varepsilon)<\infty.
\end{equation} Let \begin{equation}\label{E.2.23}  z_n(\varepsilon) = P(X_n<\varepsilon).\end{equation} By Lemma
\ref{L.2},

\begin{equation}\label{E.2.24}
r_n(2\varepsilon):=P(X_n\geq 2\varepsilon)\geq \frac{E[X_n]-2\varepsilon}{1-2\varepsilon}.
\end{equation}
\bigskip
To obtain a lower bound for $E[X_n]$ we first note that
\begin{eqnarray}\label{E.2.25}
&& E[X_{n+1}1_{X_{n+1}\geq \varepsilon}]\geq \int_0^1\int_0^1\frac{x+y}{2}1_{x+y\geq
2\varepsilon}F_n(dx)F_n(dy)\cdot p(x,y,n)\\&&\qquad\qquad\qquad\quad\;\;\geq
\int_{\varepsilon}^1\int_{\varepsilon}^1 xF_n(dx)F_n(dy)\cdot(1-q_n(\varepsilon))\nonumber\\
&&\qquad\qquad \qquad\quad\;\;= (1-z_n(\varepsilon))E[X_n 1_{X_n\geq
\varepsilon}]\cdot(1-q_n(\varepsilon)),\nonumber
\end{eqnarray}
where $F_n(dx)$ denotes the distribution of the  random variable $X_n$.
 Therefore for $n>n_0$ (to be taken sufficiently large),
\begin{eqnarray}\label{E.2.26}
 &&\quad E[X_n]\geq
E[X_n1_{X_n\geq\varepsilon}]\\&&\qquad\qquad=\prod_{k=n_0}^{n-1}(1-z_k(\varepsilon))(1-q_k(\varepsilon))E[X_{n_0}1_{X_{n_0}\geq
\varepsilon}].\nonumber \end{eqnarray}

From (\ref{E.2.18}), (\ref{E.2.23}), (\ref{E.2.24}),
\begin{eqnarray}\label{E.2.27}
 \quad && z_{n+1}(\varepsilon) \;   =P(X_{n+1}<\varepsilon)\\
 &&\qquad\qquad\leq(P(X_{n}<\varepsilon))^{2}+2P(X_{n}<\varepsilon)P(\varepsilon<X_{n}%
\leq2\varepsilon)+(P(\varepsilon<X_{n}\leq 2\varepsilon))^{2}q_{n}%
(\varepsilon)\nonumber\\
&&\qquad\qquad
\leq(P(X_{n}<\varepsilon))^{2}+2P(X_{n}<\varepsilon)(1-r_{n}(2\varepsilon))+(1-r_{n}(2\varepsilon)
)^{2}q_{n}(\varepsilon)\nonumber\\
&& \qquad\qquad \leq z_{n}^{2}(\varepsilon)+z_{n}(\varepsilon)2(1-r_{n}(2\varepsilon))+q_{n}(\varepsilon)(1-r_{n}(2\varepsilon))^{2}\nonumber\\
&&\qquad\qquad
=z_{n}(\varepsilon)(z_{n}(\varepsilon)+2(1-r_{n}(2\varepsilon)))+q_{n}(\varepsilon)(1-r_{n}(2\varepsilon))^{2}\nonumber\\
&&\qquad\qquad \leq
z_{n}(\varepsilon)(z_{n}(\varepsilon)+2(1-r_n(2\varepsilon)))+q_n(\varepsilon)\nonumber
\end{eqnarray}
where we have used $P(\varepsilon<X_n\leq 2\varepsilon)\leq 1- r_n(2\varepsilon)$.
\medskip

In order to prove that $\liminf E[X_n] >0$ for sufficiently large $C_2$, from (\ref{E.2.26})
it suffices to verify that we can choose  $\varepsilon, n_{0},z_{n_{0}},E[X_{n_{0}}1_{X_{n_0}\geq \varepsilon}]$ and $C_2$ such that%

\begin{eqnarray}\label{E.2.28}\quad
 \liminf_{n\to\infty} \prod_{k=n_0}^{n-1}(1-z_k(\varepsilon))(1-q_k(\varepsilon))E[X_{n_0}1_{X_{n_0}\geq
\varepsilon}]>0. \end{eqnarray}

Suppose that \begin{equation}\label{E.2.29} z_{n_0}(\varepsilon)\leq \frac{\varepsilon}{2} \text{
 and  }C_2\text{ is large enough so that }q_{n_0}(\varepsilon)<\frac{\varepsilon^2}{4}\end{equation}
(see (\ref{E.2.20}),(\ref{E.2.21})). Note that $C_2$ may depend on $\alpha$ (see (\ref{E.2.20})).
Assume that
\begin{equation}\label{E.2.31}
z_{m}+2(1-r_m(2\varepsilon))<s, \text{  for }  n_0\leq m\leq n,
\end{equation} for some $s\in (0,1)$ such that   \begin{equation}\label{E.2.30} s+\frac{\varepsilon}{2(1-s)}\leq 1.\end{equation}
  Then it follows from (\ref{E.2.27}),(\ref{E.2.29}) that
 for  $n\geq n_0$,
\begin{equation}\label{E.2.32cor}
z_{n+1}(\varepsilon)\leq z_{n_0}s^{n-n_0+1}+\sum_{k=n_0}^n s^{n-k}q_k(\varepsilon)\leq
z_{n_0}s^{n-n_0+1}+ \frac{q_{n_0}(\varepsilon)}{1-s}\leq \frac{\varepsilon}{2},\;,
\end{equation}

 Then by (\ref{E.2.32cor}),  $\{z_n\}_{ n_0\leq m\leq n}$ are bounded by the terms of a summable sequence, namely  (see (\ref{E.2.22})),
\begin{eqnarray}\label{E.2.32}\quad &&\sum_{n=n_0}^\infty z_n(\varepsilon,s)  =
\frac{\varepsilon}{2}\sum_{n=n_0}^\infty s^{n-n_0}+\sum_{n=n_0}^\infty \sum_{k=n_0}^n
s^{n-k}q_k(\varepsilon)\\&&\qquad\qquad\qquad=\frac{\varepsilon}{2(1-s)}+\sum_{k=n_0}^\infty\sum_{n=k}^\infty
s^{n-k}q_k(\varepsilon)=\frac{\varepsilon}{2(1-s)}+ \frac{1}{1-s}\sum_{k=n_0}^\infty
q_k(\varepsilon) <\infty.\nonumber
\end{eqnarray}

\bigskip

We first choose $2\varepsilon =0.1$ and $s=0.775$ which satisfies (\ref{E.2.30}). Then by
(\ref{E.2.24}), $2(1-r_n(2\varepsilon))<0.75$ and $z_{n_0}+2(1-r_n(2\varepsilon))<s$   provided
that $E[X_m]> 2/3$ for $n_0\leq m\leq n$.
 We now choose $n_0$
sufficiently large  so that
\[  \prod_{k=n_0}^{n-1}(1-z_k(\varepsilon,s))(1-q_k(\varepsilon)) > 0.9,\]
and $C_2$ sufficiently large so that $E[X_{n_0}1_{X_{n_0}>\varepsilon}]\geq 3/4$ and $z_{n_0}\leq
{\varepsilon}/{2}$. Note that by choosing $C_2$ sufficiently large we have $z_{n_0}=0$ and
$E[X_{n_0}1_{X_{n_0}>\varepsilon}]=1$, since  $X_{n_0}$ is atomic and positive. By continuity, this
can also be done for connection probabilities strictly less than $1$ but sufficiently close to $1$.
We then can verify from (\ref{E.2.26}) that  $E[X_n]\geq E[X_{n}1_{X_{n}>\varepsilon}]\geq
0.9\times {3}/{4}>2/3$, so we have the consistency condition $E[X_n]>2/3$  for all $n\geq n_0$, and
therefore (\ref{E.2.31}) holds for all $n\geq n_0$.

\bigskip

If the two $n$-balls in the $(n+1)$-ball have density $\geq \varepsilon$, then by considering the
events that the two $n$-balls are connected, and that they are not connected, it is easy to show
that
\begin{equation}\label{E.2.33}
E[X_{n+1}]\geq  E[X_n](1-q_n(\varepsilon))  +q_n(\varepsilon)
\frac{E[X_n]}{2}=E[X_n](1-\frac{q_n(\varepsilon)}{2}),\end{equation} which together with
(\ref{E.2.10}), (\ref{E.2.22}) and  Lemma \ref{L.3} proves (\ref{E.2.15})
($\{E[X_n]\}$ is a Cauchy sequence), and also
\begin{equation}\label{E.2.34} Var[X_{n+1}]\leq \frac{1}{2}E[X^2_n]+\frac{1}{2}(E[X_n])^2 - (E[X_n])^2(1-\frac{q_n(\varepsilon)}{2})^2
\leq \frac{1}{2}Var[X_n]+q_n(\varepsilon), \end{equation} which proves  (\ref{E.2.16}) and  then (\ref{E.2.33}), (\ref{E.2.34})
prove (\ref{E.2.17}).\quad $\blacksquare$

\strut

We now modify the argument with $N=2$  to prove the theorem for general $N$. To prepare, we begin
with some lemmas.


\begin{lemma}\label{L.con}  Consider the Erd\H{o}s-R\'enyi graph $\mathcal{G}(N,1-q)$.   Let $P(N,q)$ denote the probability that the
graph is connected.  Then as $q\to 0$, \begin{equation}\label{E.2.35}  P(N,q)\geq 1-
C(N)q^{N-1},\end{equation} where $C(N)$ is a constant such that $C(N)\sim N$ for large $N$
\end{lemma}
\begin{proof}  Recall the basic formula (\cite{Gil}, equation (4), also see  \cite{Bo}, p. 198)
\begin{equation}\label{E.2.36}  P (N,q)=1-\sum_{k=1}^{N-1} \left(  \begin{array}{c}N-1 \\k-1\end{array}
 \right)P(k,q)q^{k(N-k)}\geq 1 -\sum_{k=1}^{N-1} \left(  \begin{array}{c}N-1 \\k-1\end{array}
 \right)q^{k(N-k)}.\end{equation}
The result follows by noting that the dominant term in $1-P(N,q)$  as $q\to 0$ is given by the
smallest power of $q$ on the right side of (\ref{E.2.36}) which is $N-1$ (from the terms $k=1$ and
$k=N-1$).
\end{proof}

\begin{corollary}\label{C.6}
 Consider the graph  $\mathcal{G}_k(N,\{x_1,\dots,x_N\})$. If $x_i\geq
\varepsilon$ for all $i$,  then for large $k$ \begin{equation}\label{E.2.37}
P(\mathcal{G}_k(N,\{x_1,\dots,x_N)\}) \text { is connected})\geq 1-
q^N_k(\varepsilon),\end{equation} where \begin{equation}\label{E.2.38} q^N_k(\varepsilon)=
C(N)(q_k(\varepsilon))^{N-1}\end{equation} with \begin{equation}\label{E.2.39}
q_k(\varepsilon)=q(N,k,\varepsilon):= e^{-C_2\varepsilon^2\, k^\alpha/N^2}\leq C_3\gamma^{k^\alpha},\end{equation}
with  constants $C_3>0$ and $0<\gamma<1$. Hence
\begin{equation}\label{E.2.39+} \sum_k q^N_k(\varepsilon)<\infty.\end{equation}
\end{corollary}
\begin{proof}  By (\ref{E.2.19}),(\ref{E.2.20}),(\ref{E.2.21}) the probability that  two $(k-1)$-balls in a $k$-ball with respective densities
 $x_1,x_2\geq \varepsilon$  are
connected is given by

\begin{equation}\label{E.2.40} 1-q_k(\varepsilon) \geq 1-e^{- {C_{2}\varepsilon^2 k^{\alpha}}/{N^2} }\quad\text{ for large }k.
\end{equation}
The result then follows by Lemma \ref{L.con}.
\end{proof}
\medskip

For the $N$ $(k-1)$-balls in a $k$-ball, define \begin{equation}p(x_1,\dots,x_N,n) =P(\text{all the
}(k-1)\text{-balls with densities }x_1,\dots,x_N \text{ are connected}).\end{equation} Then by
Corollary \ref{C.6} and analogously as in the case $N=2$ (see (\ref{E.2.25})),

\begin{eqnarray}\label{E.2.41}\quad
&&E[X_{n+1}1_{X_{n+1}\geq \varepsilon}]\;\geq  \int_0^1\dots\int_0^1\frac{\sum_{i=1}^N x_i}{N}
1_{\sum x_i\geq N\varepsilon}\prod_{i=1}^N F_n(dx_i)\cdot
p(x_1,\dots,x_N,n)\\&&\qquad\qquad\qquad\qquad\geq \frac{1}{N}\sum_{i=1}^N
\int_{\varepsilon}^1\dots \int_{\varepsilon}^1 x_i F_n(dx_i)\prod_{j\ne i} F_n(dx_j)\cdot(1-q^N_n(\varepsilon))\nonumber\\
&& \qquad\qquad\qquad\qquad= (1-z_n(\varepsilon))^{N-1}E[X_n 1_{X_n\geq
\varepsilon}]\cdot(1-q^N_n(\varepsilon)),\nonumber
\end{eqnarray}
where
\begin{equation}\label{E.2.42} z_n(\varepsilon)=P(X_n< \varepsilon).
\end{equation}

 Therefore for $n>n_0$,
\begin{equation}\label{E.2.43}
 E[X_n]\geq
E[X_n1_{X_n\geq\varepsilon}]=\prod_{k=n_0}^{n-1}(1-z_k(\varepsilon))^{N-1}(1-q^N_k(\varepsilon))E[X_{n_0}1_{X_{n_0}\geq
\varepsilon}]. \end{equation}

Let
\begin{equation}\label{E.2.44} r_n(N\varepsilon)=P(X_n\geq N\varepsilon).\end{equation}
Then by Lemma \ref{L.2},
\begin{equation}\label{E.2.45}
r_n(N\varepsilon)\geq \frac{E[X_n]-N\varepsilon}{1-N\varepsilon}.
\end{equation}

\begin{lemma}
\begin{equation}\label{E.2.46} P(X_{n+1}<\varepsilon)\leq
NP(X_n<\varepsilon)(1-r_n(N\varepsilon))^{N-1}+q^N_n(\varepsilon).\end{equation}
 \end{lemma}

\begin{proof}
We first note that if the density of one of the $n$-balls in the $(n+1)$-ball is larger than
$N\varepsilon$ then $X_{n+1}>\varepsilon$.  Second, if the densities of all the balls  are larger
than $ \varepsilon$ and the balls  are connected, then $X_{n+1} > \varepsilon$. Therefore
\begin{align*}\{X_{n+1}<\varepsilon\}\subset & \{\text{densities of all }n\text{-balls }<
N\varepsilon\}\cap \\& \left[\{\text{densities all }n\text{-balls }> \varepsilon\text{ and not
connected}\}\cup \{\text{density of at least one }n\text{-ball }<\varepsilon\}\right].\nonumber
\end{align*}

Then

\begin{align*}P(X_{n+1}<\varepsilon)\leq &
NP(X_n<\varepsilon)(P(X_n< N\varepsilon))^{N-1}\\&+(P(\varepsilon\leq X_n<
N\varepsilon))^{N}P(\mathcal{G}_n(N,\{\varepsilon,\dots,\varepsilon\})\text{  not connected})\\
&\leq NP(X_n<\varepsilon)(P(X_n< N\varepsilon))^{N-1}+(P(\varepsilon\leq X_n<
N\varepsilon))^{N}q^N_n(\varepsilon)\\
&\leq NP(X_n<\varepsilon)(1-r_n(N\varepsilon))^{N-1}+(1-r_n(N\varepsilon))^{N}q^N_n(\varepsilon)
\\&\leq
NP(X_n<\varepsilon)(1-r_n(N\varepsilon))^{N-1}+q^N_n(\varepsilon).
 \end{align*}
The first summand on the right corresponds to the case that at least one density
 $<\varepsilon$ and all densities  $<N\varepsilon$; the second summand corresponds to the case in which all densities
 $x_i$ are in $ [\varepsilon,N\varepsilon)$ and the balls not connected.
\end{proof}
\bigskip

From (\ref{E.2.41}),(\ref{E.2.45}), \begin{equation}\label{E.2.47} z_{n+1}(\varepsilon) \leq
Nz_n(\varepsilon)(1-r_n(N\varepsilon))^{N-1}+q^N_n(\varepsilon).\end{equation}

\bigskip
\noindent Using (\ref{E.2.40})-(\ref{E.2.47}) we proceed analogously as in the case $N=2$. We can
then choose $\varepsilon={0.1}/{N}, E[X_{n_0}1_{X_{n_0}\geq\varepsilon}]\geq {3}/{4}, $ so that
(using inequality (\ref{E.2.45})),
 \begin{equation}\label{E.2.51+} N(1-r_n(N\varepsilon))^{N-1}\leq
2(1-r_n(N\varepsilon))<0.75,\end{equation} provided that $E[X_{n}1_{X_{n}\geq\varepsilon}]\geq {2}/{3}$.
Finally, we can then choose $ n_0,z_{n_0}$ and $C_2$ so that the sequence $\{ z_n(\varepsilon)\}$
is summable as in the case $N=2$ and we have
\[ \prod_{k=n_0}^{\infty}(1-z_k(\varepsilon))^{N-1}(1-q^N_k(\varepsilon))>0.9\] so that  $E[X_n1_{X_n\geq \varepsilon}]\geq {2}/{3}$
for all $n\geq n_0$ as in the case $N=2$. \quad $\blacksquare$

\begin{remark}  In the case $\alpha >1$, combining (\ref{E.2.47}),  (\ref{E.2.51+}) with (\ref{E.2.39}) we obtain
\begin{equation}\label{E.2.52}
z_n(\varepsilon)\leq c\,\zeta^n\end{equation} with constants $c>0,\; 0<\zeta<1$.
\end{remark}
\bigskip

\noindent \textbf{Proof of Lemma \ref{L.4} for general $N$.} This follows from (\ref{E.2.10}) and
the next inequalities which are  analogous to (\ref{E.2.33}), (\ref{E.2.34}), that can be proved
similarly to the case $N=2$  by considering the events that all the $N$ $n$-balls are connected, or
not, and using (\ref{E.2.37}), (\ref{E.2.38}), (\ref{E.2.39}), (\ref{E.2.39+}):
\begin{equation}\label{E.2.48} E[ X_{n+1}]\geq E[X_n] +O(q^N_n(\varepsilon)),\end{equation}
and
\begin{equation}\label{E.2.49} Var[X_{n+1}]\leq \frac{1}{N}Var[X_n]
+O(q^N_n(\varepsilon))\end{equation} as $n\to\infty$.

\bigskip

The proof of percolation  is then finished as in the case $N=2$ using the previous formulas, and
the uniqueness follows from Theorem 1.2 in \cite{KMT}. This completes the proof of Theorem
\ref{T.1}. $\blacksquare$

\section{Properties of the percolation clusters}

In this section we will obtain some properties of the percolation clusters that will be used for
studying behaviour or random walks on the clusters. We  will  use parts of the scheme of
\cite{DGo2} referred to in the introduction in the case $\delta =1$, that is,
\begin{equation}\label{E.3.1} k_n=\lfloor K n\log n\rfloor,\; n=1,2,\dots, \end{equation}
\begin{eqnarray}
\label{E.3.2} c_{k_n} &=&  C+a \log n \cdot n^{b \log N},
\end{eqnarray}
$K>0$, $C\geq 0$, $a>0$, $b> 0$, $c_{k_n} \leq c_k \leq c_{k_{n+1}}$ for $k_n<k<k_{n+1}$.
(\ref{E.3.1}) implies that
\begin{equation}\label{E.3.3}  k_{n+1} -k_n\sim K\log n\text{   as  }n\to\infty.\end{equation}

\subsection{Cutsets for $\delta =1$}

A {\em cutset} of a graph is a set of edges in the graph which, if removed, disconnects the graph.
\medskip

We consider the percolation cluster of $\Omega_N$ in the case $\delta=1$ with $C_2>0$. We will
construct a sequence of cutsets for the cluster that will be used to prove recurrence of the random
walk on the cluster in the case $\alpha \leq 1$.

The following argument holds with $K=1$
 (in the special case $N=2$ we need a minor modification which we omit here).  First recall that by
 \cite{DGo2} (Lemma 5.2 with $K=1$) we have the following result (this does not need the condition
 $2/\log N<K<b)$.

\begin{lemma}\label{L.8} For $0<b<2-{1}/{\log N}$,
 with probability one there exists $n_0$ such that  for all $n\geq n_0$ there is no skipping over two successive annuli
$(k_n,k_{n+1}]$, that is, there are no single edge connections between the annulus $(k_{n-1},k_n]$
and the annuli $(k_{n+2},k_{n+3}],(k_{n+3},k_{n+4}],\;  etc$.

\end{lemma}

\begin{lemma}\label{L.9} For any $\alpha >0$ there exists  a sequence of finite cutsets $\Pi_j,\,j\geq1$, for the
percolation cluster that are pairwise disjoint for large $j$, and such that
\begin{equation}\label{E.3.4} E|\Pi_j|\leq \frac{\kappa_j}{N}\quad \text{for large }j,
\end{equation}
where
\begin{equation}\label{E.3.5} \kappa_j=a2^\alpha\log j \cdot j^\alpha,
\end{equation}
with $a$ as in (\ref{E.3.2}).

\end{lemma}
\begin{proof} We take $b$ so that $0<b<\min(\alpha/\log N,2-1/\log N)$.

Let $I_j =(k_{2j},k_{2(j+1)}]\text{-annulus}$, then by Lemma \ref{L.8} $I_j$ is connected by edges
only to $I_{j-1}$ and  $I_{j+1}$ for large $j$.

Note that for $ 2(j+1)<\ell\leq 2(j+2)$,
\[ c_\ell \lesssim \kappa_j \text{   for large  } j,\]
where $\kappa_j$ is given by (\ref{E.3.5}).  For a vertex $x\in I_j$, let
\[\mathcal{M}_j(x)=\{\text{vertices in } I_{j+1}\text{  connected to }x\text{  by an edge}\}.\]
Then
\[
|\mathcal{M}_j(x)|=\sum_{\ell=k_{2(j+1)}+1}^{k_{2(j+2)}}\text{Bin}\left(N^\ell-N^{\ell-1},c_\ell/N^{2\ell}\right).\]
By ultrametricity, the distribution of $|\mathcal{M}_j(x)|$ is the same for any  $x\in I_j$. Then
by (\ref{E.3.3})
\begin{eqnarray*}
&&E|\mathcal{M}_j(x)|
=(1-\frac{1}{N})\sum_{\ell=k_{2(j+1)}+1}^{k_{2(j+2)}}\frac{c_\ell}{N^\ell}\\&& \sim \kappa_j
\left(\frac{1}{N^{k_{2(j+1)}+1}}-\frac{1}{N^{k_{2(j+2)}+1}}\right)\nonumber\\&&\sim
\frac{\kappa_j}{NN^{k_{2(j+1)}}}\text{   for large  }j.\nonumber\end{eqnarray*} Let
\[ \Pi_j=\{\text{edges connecting vertices in }I_{j}\text{ restricted to the cluster and vertices in }I_{j+1}\}.\]
Then the sets $\Pi_j$ are finite cutsets for the cluster and they are pairwise disjoint for large
$j$. Hence
\[ E|\Pi_j|\leq |I_j|E|\mathcal{M}_j|\lesssim N^{k_{2(j+1)}}\frac{\kappa_j}{NN^{k_{2(j+1)}}}\lesssim\frac{\kappa_j}{N}\text{   for large  }j.\]

\end{proof}

\subsection{Graph diameters and path lengths}\label{ss.3.2}

\noindent In this part we will obtain bounds for the lengths of paths in the percolation clusters joining two
points within distance $k_n$ from {\bf 0} for large $n$, for $\delta=1$ and $\delta <1$, which are of independent interest.
This will be done by means of known results on diameters of random graphs \cite{CL,RW}. However, for the proof of transience of
the random walk on the cluster in the case $\delta =1$ we will need a stronger result with a probability bound.

We assume that $n_0$ is large enough according to the proofs of Theorems 3.1(b) and 3.5(b) in
\cite{DGo2} (we will  refer to parts of those proofs). This means that the things we will do are
possible for $n\geq n_0$, in particular there exist the direct edge connections between clusters we
will refer to. If the two points are in the same $k_{n_0}$-cluster, the length of a path joining
them is bounded by the diameter of the cluster. Therefore we will assume
that the two points lie in the clusters of different $k_{n_0}$-balls. We proceed as follows:

\noindent $\bullet$ Find bounds for the diameters of the Erd\H{o}s-R\'enyi random graphs $G({\cal
N}_n,p_n)$ defined below, whose vertices are $k_n$-balls in a $k_{n+1}$-ball, and the connection
probability $p_n$ is defined in terms of direct edges  between the clusters of those
$k_n$-balls. These graphs are also ultrametric random graphs.

\noindent $\bullet$ Since the $k_n$-balls consist of $k_{n-1}$-balls,  find bounds for the length
of a path of $k_{n-1}$-balls in a $k_n$-ball connecting an incoming $k_{n-1}$-ball and an outgoing
$k_{n-1}$-ball (this may be called a $k_n${\em-level path}). Such a path may visit a $k_{n-1}$-ball
more than one time, but that does not matter because we consider shortest paths.

\noindent $\bullet$ Having done the previous two things, do the same  going from $n$ to $n-1$,
etc., down to $n_0$, where we end up with  the diameters of the clusters of $k_{n_0}$-balls which are i.i.d., and we denote their expected value
by $D(n_0)$.

In the arguments and calculations for path lengths we may think of paths
 within the ball $B_{k_n}(\mathbf{0})$
 joining {\bf 0} to a point in the $(k_{n-1},k_n]$-annulus. However, by ultrametricity
 any point in $B_{k_n}(\mathbf{0})$ is a center, so the bounds hold as well for paths joining any two points within distance $k_n$ from {\bf 0}.

We recall from \cite{DGo2} (Def. 4.1, Def. 5.5) that a $k_n$-ball is ``good'' if its cluster has
size at least $N^{\gamma k_n}$ for $\delta<1$, where $(1+\delta)/2<\gamma<1$, and if its cluster
has size at least $\beta N^{k_n}$ for $\delta=1$, with some  $0<\beta<1$. In the case $\delta=1$ we
assume $b>K>2/\log N$, which corresponds to $\alpha >2$. Under these conditions, in the proofs of
Theorems 3.1(b) and 3.5(b) it is shown that for all but finitely many $n$ the $k_n$-balls in any
increasing nested sequence  are good (see \cite{DGo2}, (4.22) for $\delta<1$, (5.11), (5.23) for
$\delta=1$).

Let ${\cal N}_n$ denote the number of good $k_n$-balls in a $k_{n+1}$-ball, and
\begin{equation}\label{E.3.6} p_n=P(\hbox{\rm the clusters in two good}\,\,k_n\hbox{\rm-balls in a}\,\,k_{n+1}\hbox{\rm-ball
are connected}).\end{equation}
 Note that $p_n$ is random because the sizes of
the clusters are random. We consider the Erd\H{o}s-R\'enyi graphs  $G({\cal N}_n,p_n)$. In all the
cases, ${\cal N}_n\to\infty$ as $n\to\infty$. In the proof of Theorem 5.3(b) in \cite{DGo2}  it is
shown that for $\delta=1$, $b\geq 1$, and  $\beta >$ 1/5, the graph $G({\cal N}_n,p_n)$ becomes
connected for large $n$. We shall see that for $b>2$ it even becomes complete (all pairs of
vertices are connected).

\subsubsection {Diameters of the  graphs $G(\mathcal{N}_n,p_n)$}

First we obtain bounds for the diameters of the graphs $G({\cal N}_n,p_n)$ in the following cases
where except in case 2 we assume that $K=1$.
\medskip

\noindent {{\em Case 1. $\delta<1$.}}

From \cite{DGo2} ((4.5), (4.8), (4.10)), we have the lower bound for (\ref{E.3.6})
$$
p_n\geq 1-{\rm exp}(-c N^{2\gamma k_n-(1+\delta)k_{n+1}})>1- {\rm exp} (-cN^{\varepsilon n\log
n})\,\,\,\,{\rm as}\,\, n\to\infty,$$ with some $0<\varepsilon<1$, hence $p_n\to 1$ and ${{\cal
N}_np_n}/{\log {\cal N}_n}\to \infty\,\,{\rm as}\,\,n\to\infty,$ therefore by Theorem 2 in
\cite{CL} diam$(G({\cal N}_n,p_n))$ is  concentrated on at most two values $\{1,2\}$ at
$$\frac{\log {\cal N}_n}{\log ({\cal N}_np_n)}\to 1\quad{\rm as}\quad
n\to\infty,$$ which implies that diam$(G({\cal N}_n,p_n))\leq 2$ for large $n$.

\medskip

\noindent {{\em Case 2. $\delta=1$, $b>2K$.}}

\begin{lemma}\label{L.10'} Assume that $b> 2K$.  Then the graph $G({\cal N}_n,p_n)$
is complete and  $\text{diam}(G({\cal N}_n,p_n))=1$ for large $n$.
\end{lemma}
\begin{proof}From \cite{DGo2} (proof of Lemma 5.7 except the last step), we have
$$p_n\geq 1-{\rm exp}(-{\beta^2a\log n\cdot N^{(b-2K)\log n}})\to 1 \text{   as }n\to\infty.$$
Then, since ${\cal N}_n\leq N^{K\log n}$,

$q_n:=P$(some pair of clusters of two good $k_n$-balls in a $k_{n+1}$-ball is not connected)

$\qquad\qquad\leq N^{2K\log n}(1-p_n)$ $\leq N^{2K\log n}{\rm exp}(-\beta^2a\log n\cdot
N^{(b-2K)\log n})$

$\qquad\qquad = n^{2K\log N}/n^{\beta^2an^{(b-2K)\log N}}$,

\noindent Therefore $\sum q_n<\infty$, hence by Borel-Cantelli for all but finitely many $n$ the
graph $G({\cal N}_n,p_n)$ is complete. So, $\text{diam}(G({\cal N}_n,p_n))=1$ for large $n$.
\end{proof}
\medskip




\noindent {{\em Case 3. { $\delta=1, b= 2$.}}}

As in case 2,
$$p_n\geq 1-{\rm exp}(-\beta^2a\log n)\to 1\,\,\,\,{\rm as}\,\,n\to\infty,$$
then as in case 1, diam$(G({\cal N}_n,p_n))$ $\leq 2$ for large $n$.
\bigskip

\noindent{{\em Case 4. {$\delta=1, 1<b<2$.}}}

 Again as above,
$$
p_n\geq 1-{\rm exp} \left(- \frac{\beta^2a\log n}{N^{(2-b)\log n}}\right)
>\frac{\beta^2a\log n}{2N^{(2-b)\log n}}
$$
for large $n$  such that
$$\frac{\beta^2a\log n}{2N^{(2-b)\log n}} < 0.7968.$$
Since ${\cal N}_n\leq N^{\log n}$,
 and from the proof of Theorem 3.5(b) in \cite{DGo2} ((5.17), (5.18) and step 1) we have
${\cal N}_n\geq \beta N^{\log n}$ for large $n$, then
$$
\frac{{\cal N}_np_n}{\log {\cal N}_n}\geq \frac{\beta^3a\log n\cdot N^{b\log n}} {2N^{\log n}\log
n\cdot \log N}= \frac{\beta^3a}{2\log N}N^{(b-1)\log n}\to\infty\,\,\,\,{\rm as}\,\,n\to\infty,
$$
then, again by Theorem 2 in \cite{CL}, diam$(G({\cal N}_n,p_n))$ is  concentrated on at most two
values at
\begin{eqnarray*}
\lefteqn{\frac{\log {\cal N}_n}{\log ({\cal N}_np_n)}\leq \frac{\log n\cdot \log N}
{\log\left(\displaystyle
\frac{\beta^3a}{2}\log n \cdot N^{(b-1)\log n}\right)}}\\
&\sim& \frac{\log n\cdot \log N}{\log(\log n\cdot N^{(b-1)\log n})}\lesssim\frac{1}{b-1}\,\,\,{\rm
as}\,\,\, n\to\infty,
\end{eqnarray*}
so, diam$(G({\cal N}_n,p_n))\leq {b}/({b-1})$ for large $n$. Note that this bound is continuous at
$b=2$ (case 3).
\medskip

\noindent{{\em Case 5. $\delta=1, b=1.$}}

By Theorem 1.2 of \cite{RW} and \cite{DGo2} (Lemma 5.7),
$$
 p_n\gtrsim \frac{\beta^2a\log n}{N^{\log n}}=\frac{\lambda_n}{{\cal N}_n}
=:\widetilde{p}_n\quad\hbox{\rm for large}\quad n,$$ where
$$
\lambda_n=\beta^2a\log n\cdot \frac{{\cal N}_n}{N^{\log n}}\leq \beta^2a\log n\leq {\cal
N}^{1/1000}_n\quad\hbox{\rm for large}\quad n,
$$
hence diam$(G({\cal N}_n,\widetilde{p}_n))$ is concentrated on two values around
$$
f({\cal N}_n,\lambda_n)=\frac{\log {\cal N}_n} {\log \lambda_n}+\frac{2\log{\cal N}_n}{\log
(1/\lambda^*_n)}+O(1) \,\,\,\,\hbox{\rm for large}\,\,n,$$ where
\begin{eqnarray*}
\lambda^*_ne^{-\lambda^*_n}&=&\lambda_ne^{-\lambda_n},\,\,\,\lambda_n\to\infty,\,\,\,\lambda^*_n\to
0,\end{eqnarray*} then
\begin{eqnarray*}
\log \lambda_n&\sim& \log\log n,\\
\log \lambda^*_n-\lambda^*_n&=&\log\lambda_n-\lambda_n,
\end{eqnarray*}
since ${\cal N}_n\geq \beta N^{\log n}$ for large $n$, then
\begin{eqnarray*}
\log (1/\lambda^*_n)&=&\lambda_n-\log\lambda_n-\lambda^*_n\gtrsim \beta^3a\log n,
\end{eqnarray*}
and
\begin{eqnarray*}
f({\cal N}_m,\lambda_n)&\lesssim&\frac{\log{\cal N}_n}{\log\log n}+ C_1\frac{\log{\cal N}_n}{\log
n}+O(1) \lesssim C_2\log n\quad \hbox{\rm for large}\quad n,
\end{eqnarray*}
so, diam$(G({\cal N}_n,{p}_n))$ $\lesssim L \log n$ for some constant $L>0$ and large $n$.

\subsubsection{Path lengths in $\mathcal{G}_N^\infty$}

    We will now obtain a bound $L(n)$ for the expected length of a path
joining two points in a $k_n$-ball for large $n$ (recall that this is the
same as the expected length of a path from $\mathbf{0}$  to a point in the
$(k_{n-1}, k_n]$-annulus). This follows the three steps in the procedure
described at the beginning of subsection 3.2, once we have results on the
expected values of the diameters on the graphs $G(\mathcal{N}_n, p_n)$. We use here expected values of the diameters of the graphs for a
general argument, but we actually found bounds for the diameters
themselves for large $n$ in the five special cases considered above. Note
however that we do the calculation by an iteration in reverse order, that
is, starting with $L(n_0) = D(n_0)$ and ending with $L(n_0 + j)$.

Let $n_0$ be large enough as mentioned above. We have denoted by
$D(n_0)$ the expected diameter of the cluster in a $k_{n_0}$-ball. Let $D(n_0+j)=E[{\rm diam}(G(\mathcal{N}_{n_0+j},p_{n_0+j}))]$, $j\geq 1$,
$L(n_0+j)$ denote  a bound for the expected length of a path joining  $\bf 0$ to a point in  the
$(k_{n_0+j-1},k_{n_0+j}]$-annulus, $j\geq 1$,
$L(n_0)=D(n_0)$.

 Then $L(n_0+1)=D(n_0)(D(n_0+1)+1)+D(n_0+1)$, because there are at most $D(n_0+1)$ edges
in the $k_{n_0+1}$-ball that join  $D(n_0+1)+1$ $k_{n_0}$-balls,
considering the two  ends of the $k_{n_0+1}$-level path (a path of $k_{n_0}$-clusters), and that
the path may enter and leave  each $k_{n_0}$-cluster from different points that are joined by a
path of length at most $D(n_0)$ in the $k_{n_0}$-cluster. So,
$$L(n_0+1)=D(n_0)D(n_0+1)+D(n_0)+D(n_0+1).$$

\noindent Similarly,
\begin{eqnarray*}
L(n_0+2)&=&L(n_0+1)(D(n_0+2)+1)+D(n_0+2)\\
&=&D(n_0)D(n_0+1)D(n_0+2)+D(n_0)D(n_0+2)+D(n_0+1)D(n_0+2)\\
&&+D(n_0)D(n_0+1)+D(n_0)+D(n_0+1)+D(n_0+2)\\
&=&(D(n_0)+1)(D(n_0+1)+1)(D(n_0+2)+1)-1,
\end{eqnarray*}
We show that
\begin{eqnarray*}
L(n_0+k)&=&\prod^k_{j=0}(D(n_0+j)+1)-1\quad{\rm for}\quad k\geq 0\end{eqnarray*} by induction:
\begin{eqnarray*}
L(n_0+k+1)&=&L(n_0+k)(D(n_0+k+1)+1)+D(n_0+k+1)\\
&=&\left(\prod^k_{j=0}(D(n_0+j)+1)-1\right)(D(n_0+k+1)+1)+D(n_0+k+1)\\
&=&\prod^{k+1}_{j=0}(D(n_0+j)+1)-1.
\end{eqnarray*}
Hence \begin{equation*} L(n_0+j)\leq \displaystyle\prod\limits^j_{k=0}(D(n_0+k)+1),\quad j\geq 1.
\end{equation*} It follows that if  $D(n_0+k)<L$ for $k>1$, with $L>1$ (cases $1,3$ and $4$ above),
then
\begin{equation}\label{E.3.7} L(n_0+j)<(D(n_0)+1)(L+1)^j\,\,\,\,{\rm for}\,\, j\geq 1.\end{equation}
 \noindent and if $D(n_0+k)<
L_1 \log k$ for $k>1$, with $L_1>0$ (case $5$ above) then \begin{equation*}
L(n_0+j)<(D(n_0)+1)(L\log j)^j\,\,{\rm for}\,\,\,\,j\geq 1\,\,\hbox{\rm and some constant
}\,\,L>0.\end{equation*}

The results that will be used for the proofs below for the cases $\delta<1$, and $\delta=1,b>2$
are:

\begin{lemma}\label{L.12} In the case $\delta<1$,
\begin{equation}\label{E.3.8}
L(n_0+j)\leq K(n_0)(L+1)^j \text{   for  }j\geq 1,
\end{equation}
where $K(n_0)=D(n_0)+1$.

\end{lemma}
\begin{proof} This follows from case 1 and (\ref{E.3.7}).
\end{proof}

This result could be made stronger by an argument of the type in the next Lemma, but it suffices as it is to prove transience of the random walk on the percolation cluster for $\delta <1$.

\begin{lemma}\label{L.13}  In the case $\delta=1, b>2K>2$,
\begin{equation}\label{E.3.9} L(n_0+j)\leq C D(n_0)j\quad\text{ for} \quad j\geq 1\quad\text{ and some constant}\quad
C>0.\end{equation}

\end{lemma}
\begin{proof} Here we also refer to the setup of Section \ref{S.perc}.
To obtain an upper bound on the expected path length between two points  $\mathbf{x},\mathbf{y}$ in a ball of diameter $k_{n_0+j}$ we find a path with edges with hierarchical
lengths $k_{n_0+j},k_{n_0+j-1},\dots, k_{n_0+1}$ corresponding to connections between disjoint $k_{n_0+j-1}$-balls, $k_{n_0+j-2}$-balls,$\dots$, $k_{n_0+1}$-balls respectively, and paths within $k_{n_0}$-balls.  We first calculate an upper bound on the expected number of edges of length $k_{n_0+j}$ needed. If the points are in the same $k_{n_0+j-1}$-ball this is $0$,
 if they are in connected $k_{n_0+j-1}$-balls this is $1$ and otherwise the length of the path is bounded by $N^{K\log n}$. The probability that two $k_{n_0+j-1}$-balls  having densities at least $\varepsilon$ are connected is given by
\begin{eqnarray*}{}
&&P\left(\text{two }k_n\text{-balls with densities }\geq\varepsilon \text{ in  } B_{k_{n+1}}(\mathbf{0})\text{  are connected}\right)\\
&&\gtrsim 1-\left( 1-\frac{a\log n\cdot N^{b\log n}}{N^{2k_{n+1}}}\right)^{\varepsilon^2N^{2k_n}}\\&&
\sim 1-\exp\left(-a\varepsilon^2\log n\cdot N^{b\log n-2(k_{n+1}-k_n)}\right)\\&&
\sim 1-\exp\left(-a\varepsilon^2\log n\cdot n^{(b-2K)\log N}\right).
\end{eqnarray*}
Therefore for $\alpha >1$, recalling (\ref{E.2.42}), \begin{eqnarray}{}
&&s_n:= P\left(\text{two }k_n\text{-balls in  } B_{k_{n+1}}(\mathbf{0})\text{  are connected}\right)\nonumber\\
&&\qquad\qquad\gtrsim 1-\exp\left(-a\varepsilon^2\log n\cdot n^{(b-2K)\log N}\right)-2 z_{k_n}(\varepsilon)\nonumber\\&&\qquad\qquad
\sim 1-\exp\left(-a\varepsilon^2\log n\cdot n^{(b-2K)\log N}\right)- c \zeta^{Kn\log n}.\nonumber
\end{eqnarray}
for some constant $c>0$ and $0<\zeta<1$ (using  (\ref{E.2.52}) with $k_n$ in the place of $n$).
Therefore the expected number of edges of length $k_{n_0+j}$ is bounded by
\begin{eqnarray}
  e_{n_0+j}= 1+N^{K\log (n_0+j)} [\exp\left( -a\varepsilon^2\log (n_0+j) \cdot (n_0+j)^{(b-2K)\log N}\right)+c \zeta^{K(n_0+j)\log (n_0+j)}].\nonumber
\end{eqnarray}
We must then connect the entrance vertices and the exit vertices in each of the $k_{n_0+j-1}$-balls and following the same procedure  the bound for the expected number of length $k_{n_0+j-1}$ edges needed in a given $k_{n_0+j-1}$ -ball is given by $e_{n_0+j-1}$, and since the random variables involved are independent we obtain that the expectation of the total number of length $k_{n_0+j-1}$ edges needed is bounded by $e_{n_0+j}\cdot e_{n_0+j-1}$. Continuing we obtain that the upper bound for the expected number of edges of length $k_{n_0+1}$ is
\[ \prod_{\ell= 1}^j e_{n_0+\ell}.\] Noting that since $b-2K>0$ and $0<\zeta<1$,
\[ \lim_{j\to\infty}\prod_{\ell= 1}^j e_{n_0+\ell}=\prod_{\ell=1}^\infty \left(1+N^{K\log (n_0+\ell)} \left[\exp\left( -a\varepsilon^2\log(n_0+1) \cdot (n_0+\ell)^{(b-2K)\log N}\right)+c\zeta^{K(n_0+\ell)\log (n_0+\ell)}\right]\right)<\infty,\]
and that the expected path length between points $\mathbf{x}$ and $\mathbf{y}$ is bounded by
\[ \sum_{m=1}^j \prod_{\ell =1}^{m}e_{n_0+\ell}D(n_0)\leq \left(\prod_{\ell=1}^\infty e_{n_0+\ell}\right)D(n_0)\,j\]
where $D(n_0)$ is the expected path length between points in a $k_{n_0}$-ball, the proof is finished.

\end{proof}
\medskip

\medskip
\section{Random walks on the percolation cluster}

\subsection{Random walks and electric circuits}
In this subsection we review briefly some basic background on random walks and electric circuits on
graphs which will then be applied to random walks on the percolation clusters.

The nearest neighbour random walk on a finite or an infinite graph such as the percolation cluster
is a Markov chain on the countable connected subset given by the graph.
 Here  there is a transition between neighbours $x$ and $y$  with probability
\[ p_{xy}=\frac{1}{n(x)},\]
where $n(x)$ is the number of neighbours of $x$ in the graph.

The random walk is a {\em reversible} since setting $\pi(x)=n(x)$ we have
\begin{equation*} C(x,y)= \pi(x)p_{xy}=\pi(y)p_{yx} \; \text{for all }\; x,y.\end{equation*}
 In this case $C(x,y)$ is called the {\em conductance} between $x$ and $y$ and the {\em resistance} $R(x,y)$ is
 defined as
$R(x,y)={1}/{C(x,y)}$.

For  any finite set $Z$ of vertices the {\em effective conductance} and {\em effective resistance}
between a point $a$ and $Z$ are defined as \begin{equation*}\label{ECR}
\mathcal{C}(a\leftrightarrow Z)= {\pi(a)} P(\tau_Z<\tau^+_a),\quad \mathcal{R}(a\leftrightarrow
Z)=1/\mathcal{C}(a\leftrightarrow Z).\end{equation*} where $\tau^+_a$ is the first time after $0$
that walk visits $a$ and $\tau_Z$ is the hitting time of $Z$.


If $G$ is an infinite connected graph, let $G_n$ be a finite subgraph of $G$ such that $G_n\uparrow
G$ as $n\to\infty$ and $Z_n:= G\backslash G_n$ (identified as a single vertex). Then the {\em
effective resistance from $a$ to $\infty$} is defined as
\[ \mathcal{R}(a\leftrightarrow\infty)=\lim_{n\to\infty}\frac{1}{ \mathcal{C}(a\leftrightarrow Z_n)}.\]

\subsection{Criteria for transience and recurrence}\label{SS.TR}

 Doyle and Snell \cite{DS} (also see \cite{LP,KL})  proved that the effective resistance is equivalent to
 the resistance computed using the laws of electric circuit theory applied to the circuit obtained
 by replacing each edge by a unit resistor resulting in the following criterion for
 transience and recurrence.
 \medskip

 \noindent \textbf{Criterion for
transience-recurrence}\label{crit-trans-rec} The random walk on an infinite connected graph is
transient, respectively recurrent, if $\mathcal{R}(a\leftrightarrow \infty)$ is finite,
respectively infinite.
\medskip

\noindent\textbf{Rayleigh monotonicity principle} Removing an edge increases the resistance between
two points. Therefore to prove that the random walk on the graph is transient it suffices to show
that it is transient on a subgraph.
\bigskip

We will use  a related criterion for transience  based on the Dirichlet's minimization principle
for energy of a flow in a circuit.

\begin{definition}
A {\em unit flow} on an infinite  graph ${G}=(V,E)$ with source $a\in V$  is a function $\theta$ on
the set of edges $E$ such that $ \theta(x,y)=-\theta(y,x)$ and  for all $x\ne a$,
\[ \sum_{x\ne a} \theta(a,x)=1\qquad \text{and   }\sum_{y\sim x}\theta(x,y)=0\;\text{for all }\; x\ne a,\]
where $x\sim y$ means that $y$ is a neighbour of $x$.
\end{definition}


\begin{definition}\label{D.energy} The  {\em energy} of the flow is
\[ \mathcal{E}(\theta)= \sum_{e\in E^*}(\theta(e))^2R(e),\]
where $E^*$ is the set of directed edges and  $R(e):= R(x,y)$ is the resistance of the  edge $e$
from $x$ to $y$.
\end{definition}



\subsubsection{Transience, finite energy criterion}\label{crit-trans}

A random walk on a countable connected graph $G$ is transient iff there is a unit flow
 from any vertex $a$  to $\infty$ on $G$  with finite energy  \cite{L}.

\subsubsection{Recurrence, Nash-Williams criterion }\label{crit-rec}

 If $\{\Pi_n\}$ is a sequence of disjoint finite cutsets
in a locally finite graph $G$, each of which separates $a$ from infinity, then
\[ \mathcal{R}(a\longleftrightarrow \infty)\geq \sum_n\left(\sum_{e\in\Pi_n}C(e)\right)^{-1}.\]
In particular, if the right-hand side is infinite, then the walk on $G$ is recurrent \cite{NW}. In
our case the edges have unit resistance, so the random walk is recurrent if
\[ \sum_n 1/|\Pi_n|=\infty.\]

\subsection{Transience and recurrence of random walks on the percolation cluster}

In this subsection we give transience and recurrence results for simple (nearest neighbour) random
walks on the percolation clusters for $\delta<1$ and for $\delta =1,\; C_2>0$.

\subsubsection{Recurrence for $\delta=1,\;\alpha\leq 1$ }

\begin{theorem}\label{T.rec}  In the case $\delta=1$, sufficiently large $ C_2$, for almost every realization of the percolation cluster the random walk
on the  cluster  is recurrent if $\alpha \leq 1$.
\end{theorem}
\begin{proof} For the cutsets $\Pi_j$ in Lemma \ref{L.9} (note that for $\alpha\leq 1$ and $N\geq 2$,
 $\alpha/\log N\leq 2-1/\log N$, see choice of $b$ at the beginning of the proof of Lemma \ref{E.3.2}).
\[ E\left(\frac{1}{|\Pi_j|}\right)\geq \frac{1}{E|\Pi_j|},\]
(by Jensen's inequality since $1/x$ is convex on $(0,\infty)$), hence by (\ref{E.3.4})
\[E\left(\frac{1}{|\Pi_j|}\right)\gtrsim \frac{N}{\kappa_j}\text{  for large }j.\]
Then by (\ref{E.3.5})
\[ E\left(\sum_j\frac{1}{|\Pi_j|}\right)\geq
N\sum_j\frac{1}{\kappa_j}=\infty,\] since $\alpha \leq 1$.

 The random variables $1/|\Pi_j|$ are
independent and bounded by $1$, hence the probability that $\sum_j1/|\Pi_j|$ diverges is positive
(\cite{Kal}, Prop. 4.14), then by Kolmogorov's $0$-$1$ law $\sum_j1/|\Pi_j|$ diverges w.p.1. Then
the recurrence of the random walk follows by the
 Nash-Williams criterion.
 \end{proof}
\bigskip

\subsubsection{Transience  for $\delta<1$}

\begin{theorem}\label{T.trans1} In the case $\delta <1$, for almost every realization of the percolation cluster  the random
walk on the cluster is transient.

\end{theorem}
\begin{proof}
Let $k_n=\lfloor n\log n\rfloor$, $c=\inf_kc_k >0$, $A_{n}= (k_{n-1},k_n]$-annulus, and denote by
${\cal M}_n$ the number of edges connecting $A_{n}$ and $A_{{n+1}}$. Then ${\cal M}_n$
stochastically dominates

$$B_n={\rm Bin}\left(|A_{n}||A_{{n+1}}|,
\frac{c} {N^{(1+\delta)k_{n+1}}}\right),$$

\noindent hence
$$E{\cal M}_n\gtrsim N^{2k_n}\frac{c}{N^{(1+\delta)k_{n+1}}}
\sim cN^{(1-\delta)n\log n}\,\,\,\,{\rm as}\,\,n\to\infty.$$

Since $cN^{(1-\delta)n\log n}>>4^n$ as $n\to\infty$, and $\text{Var}[B_n]=O(E[B_n])$ as
$n\to\infty$, it can be shown using \cite{Kal} (Lemma 4.1) that
\begin{equation}\label{E.4.2} P(\mathcal{M}_n>4^n)\to 1\text{   as   } n\to\infty.\end{equation}
 Therefore by (\ref{E.4.2})
for all large $n$ we can pick $4^n$ direct edges from $A_{{n-1}}$ to $A_{{n}}$. Since $|A_{n}|\geq
(1-\varepsilon)N^{n\log n}$ for some $0<\varepsilon<1$, we can subdivide $A_{n}$ into $4^n+4^{n+1}$
disjoint subsets, each containing $O(N^{n[\log n-\log 4/\log N]})$ vertices. We assign $4^n$ of
these subsets as {\it entrance-sets} for edges from $A_{{n-1}}$, and $4^{n+1}$ of them as {\it
exit-sets} for edges to $A_{{n+1}}$, and identify an edge from each one of the $4^n$ exit-sets in
$A_{{n-1}}$ to a different entrance set in $A_{{n}}$. The end-points of those edges are an {\it
in-vertex} in the entrance-set, and an {\it out-vertex} in the exit-set in the previous annulus.

We now connect by an edge  each one of the $4^n\cdot 4^{n+1}$ pairs (entrance-set, exit-set) in
$A_{{n}}$. The probability that there is no such a pair connection is
$$\sim\left(1-\frac{c}{N^{(1+\delta)n\log n}}
\right)^{N^{2n\log n}}\sim{\rm exp}(-cN^{(1-\delta)n\log n})\,\,{\rm as}\,\,n\to\infty,$$ hence the
probability of the event that any of the pairs fail to be connected is
$$O(4^{2n}{\rm exp}(-cN^{(1-\delta)n\log n}))\,\,\,\, {\rm as}\,\,\, n\to\infty.$$
Since this is summable, then by Borel-Cantelli there exists a random $n_0$ such that for all $n\geq
n_0$ all the pairs in $A_{n}$ are connected. Therefore we can construct an infinite tree which is
rooted at some vertex in $A_{{n_0}}$, whose nodes are in-vertices in  entrance-sets in successive
annuli $A_{n}$, such that each node has 4 children, one in each of 4 different entrance-sets in the
next annulus (all disjoint), which are connected by edges to corresponding 4 different exit-sets in
the previous annulus, so that the $n$-{\it th} generation consists of $4^n$ vertices.

It remains to connect by paths within each entrance-set its in-vertex to the end-points of the
edges that connect the entrance-set to its corresponding 4 exit-sets in the annulus (these
end-points are {\it out-vertices} in the entrance set), and similarly to connect by paths within
each exit-set its out-vertex to the end-point of the edge from its corresponding entrance-set (this
end-point is an {\it in-vertex} in the exit-set). In order for such paths to exist we need to
assume that all this is done within the percolation cluster. Since the clusters in good $k_n$-balls
have size at least  $ N^{\gamma k_n}$ with $(1+\delta)/2<\gamma<1$ for all sufficiently large $n$
(see \cite{DGo2}, Def. 4.1, (4.4), (4.6), (4.22)), we  take $N^\gamma$ instead of $N$ above, so
that the construction takes place in the percolation cluster. Then the connecting paths exist and
there may be more than one in each set. For $n>n_0$, the length of a path from a node of the tree
in $A_{{n}}$ to any one of its children in the next generation is bounded by
$1+1+K(n_0)(3^{n-n_0}+3^{n-n_0})$, with $K(n_0)$ given in Lemma \ref{L.12}. The 1's come from the
single edges between an annulus $A_{n}$ and the next one, and from the single edges connecting
out-vertices in entrance-sets to in-vertices in exit-sets in the annulus. The $3^{(n-n_0)}$'s come
from the lengths of paths joining the in-vertex to the out-vertices in an entrance set, and the
length of a path joining the in-vertex and the out-vertex in an exit-set in $A_{n}$. Hence the
length of a path from a node in the tree in the $n$th generation to any of its 4 children in the
next generation is bounded by $C3^n$, by  Lemma \ref{L.12} for some positive constant $C$.

By construction the paths joining a node in the tree to its children in the next generation can
have common edges only within the entrance-set. Since there are at least 1 and at most 4
out-vertices in an entrance-set, then each edge in the paths is used at most 4 times. Then it
follows by Proposition 3 of \cite{GKZ}, with $\beta=3$ and $\alpha=\gamma=4$, that the resistance
of the tree from the root to infinity is at most
\begin{equation*} 4\sum_{n=n_0}^\infty \frac{C3^n}{4^n}<\infty.\end{equation*}
Therefore, by the criterion for transience-recurrence in subsection \ref{crit-trans-rec}    the
walk on the percolation cluster is transient.
\end{proof}

\subsubsection{Transience for $\delta = 1$, $\alpha >6$.}

\begin{theorem}\label{T.trans2} In the case $\delta=1$, sufficiently large $C_2$ and $\alpha >6$, for almost every realization of the percolation cluster   the
random walk on the cluster is transient.
\end{theorem}
\begin{proof} The idea of the proof is to construct a subgraph of the percolation cluster
$\mathcal{C}$ and  a flow on the subgraph  that satisfies the finite energy condition criterion.

We assume that $k_n$ is given as in (\ref{E.3.1}) with  $2/\log N<K<b<\alpha/ \log N$, and in this
proof we take $b>3K$. Hence we  have the condition of Lemma \ref{L.10'}.

Given $\alpha >1$ and sufficiently large $C_2$, there exists $n_{00}$ such that for all $n\geq
n_{00}$ all the $N^{K\log n}$ $k_n$-balls in $B_{k_{n+1}}(\mathbf{0})$ are $\beta$-good with $\beta
=\varepsilon$ where $\varepsilon $ is as in the proof of Theorem 2.1. This follows since the
probability that a $k_n$-ball is not $\varepsilon$-good (i.e. $X_{k_n}<\varepsilon$) is $<
c\zeta^{Kn\log n}$ (see (\ref{E.2.42}), (\ref{E.2.52})), and
\begin{equation*} \sum_n N^{K\log n}\cdot \zeta^{Kn\log n} <\infty.\end{equation*}

The edges of the subgraph will be decomposed into a sequence of subsets:
\begin{itemize}
\item edges connecting successive  $A_n=(k_{n},k_{n+1}]\text{-annulus },\; n=1,2,\dots,\; k_n=\lfloor Kn\log
n\rfloor$,
\item these edges go from disjoint good $k_n$-balls in $A_n$ to disjoint good $k_{n+1}$-balls in $A_{n+1}$,
\item there are also edges within  the $k_n$-balls (and $k_{n+1}$-balls)   connecting the entrance and exit vertices in these balls.
\end{itemize}

\medskip

Recall that the graphs $G({\cal N}_n,p_n)$ in subsection \ref{ss.3.2} are complete for $b>2K$ and $n\geq n_0$ (Lemma
\ref{L.10'}). We now compute a lower bound for
\begin{equation}\label{E:rn1} r_n:=P(\text{a good } k_n\text{-ball in }  B_{k_{n+1}}{(\mathbf{0})}
\text{ is connected to  a good }k_{n+1}\text{-ball in }  A_{n+1})
\end{equation}
   for large $n$. We have for large $n$, from equation
   (\ref{E.3.2}),

\begin{eqnarray*}
\lefteqn{\qquad r_n\,\,\gtrsim\,\, 1-\left(1-\frac{a\log n\cdot N^{b\log
n}}{N^{2k_{n+2}}}\right)^{\varepsilon^2N^{k_n+k_{n+1}}}}\\
&&\sim 1-{\rm exp}(-a\varepsilon^2\log n\cdot N^{b\log n-[2(k_{n+2})-k_n-k_{n+1}]}),\end{eqnarray*}
and using (\ref{E.3.3})
\begin{equation*}2k_{n+2}-k_n-k_{n+1}
=k_{n+1}-k_n+2(k_{n+2}-k_{n+1})\sim 3K\log n,\end{equation*} hence
 \begin{equation} \label{E:rn2} r_n\gtrsim
1-{\rm exp}(-a\varepsilon^2\log n\cdot n^{(b-3K )\log N}).
\end{equation}
There are $ N^{K\log n}\,\,k_n$-balls in $B_{k_{n+1}}(\mathbf{0})$, and $(N^{K\log(n+1)}-1)\sim
N^{K\log(n+1)}\,\,k_{n+1}$-balls in $A_{n+1}$, and \[\sum_n n^{K\log N}(n+1)^{K\log N}{\rm
exp}(-a\varepsilon^2\log n\cdot n^{(b-3K)\log N})<\infty\quad{\rm if}\quad b>3K,\] which we now assume,
so, for such $b$ and for almost every realization of the percolation cluster there exists $n_0$
such that for all $n\geq n_0\geq n_{00}$,
\begin{equation}\label{E.4.3} \text{ every good }k_n\text{-ball in } B_{k_{n+1}}({\bf 0})\text{ is connected in $\mathcal{C}$ to every good }k_{n+1}\text{-ball in
}A_{n+1},\end{equation} and this will be used in the construction below.
\medskip

By the Rayleigh monotonicity principle, to prove that the random walk is transient on $\mathcal{C}$
it suffices to show that it is transient on a subgraph of $\mathcal{C}$. Given a realization of the
cluster and associated $n_0$ satisfying (\ref{E.4.3}), we will construct a subgraph and a unit flow
on it to satisfy  the  energy criterion for transience.

The flow has the following properties:
 Start with $B_{k_{n_0}}(
 \mathbf{0})$ and assume that {\bf 0} belongs to $\mathcal{C}$. Choose one edge
from $B_{k_{n_0}}(\mathbf{0})$ to each one of the $N^{K\log (n_0+1)}\,\,k_{n_0}$-balls in
$A_{n_0}$. The unit flow entering at {\bf 0} is divided into $N^{K\log (n_0+1)}$ equal parts going
to each one of the $k_{n_0+1}$-balls in $A_{n_0}$. The ball $B_{k_{n_0}}(\mathbf{0})$ has an
internal structure which is the set of points (vertices) of $\Omega_N$ and edges that are contained
in $\mathcal{C}\cap B_{k_{n_0}}$. There are many ways that the flow can go through paths from {\bf
0} to the (at least 1 and at most $N^{K\log (n_0+1)}$) exit-vertices in the cluster of
$B_{k_{n_0}}(\mathbf{0})$, splitting appropriately at branch vertices on the paths in order to
achieve the division of the flow as stated. Denote by $E_0$ the energy of the flow on the subgraph
of the cluster connecting $\mathbf{0}$ to  the exit vertices of $B_{k_{n_0}}(\mathbf{0})$. The flow
will then pass through a series of disjoint subsets of edges in $\mathcal{C}$  denoted
$\{G_n\}_{n\geq 1}$ with the energies denoted by $\{E_n\}_{n\geq 1}$.  $G_1$ consists of edges from
the at most $N^{K\log (n_0+1)}$ exit-vertices  in the cluster of $B_{k_{n_0}}(\mathbf{0})$ to the
$N^{K\log (n_0+1)}$ disjoint $k_{n_0}$-balls in $A_{n_0}$ and the edges connecting the entrance
vertices in these balls to the exit vertices. Similarly for $n\geq 2$ $G_n$ consists of edges from
the $N^{K\log n}$ disjoint $k_{n-1}$-balls in $A_{n-1}$  to the $N^{K\log (n+1)}$ disjoint
$k_{n}$-balls in $A_{n}$ and the edges connecting the (at most $2$) entrance vertices in these
balls to the (at most 3) exit vertices.

We now specify in detail the choice of the edges in $G_n$ and the flow in each of these edges.  For
$n>n_0$ each of the $N^{K\log (n+1)}\,\, k_{n}$-balls in $A_{n}$ gets $1/N^{K\log (n+1)}$ amount of
flow entering through $1$ or $2$ edges, which then goes through the internal structure of each
$k_{n+1}$-ball and is then divided along $1$, $2$ or $3$ edges to the $N^{K\log
(n+2)}\,k_{n+2}$-balls in $A_{n+1}$. To do this we first enumerate the $N^{K\log (n+1)}\,\,
k_{n}$-balls, in $A_{n}$ denoted $B^n_1,\dots,B^n_{N^{K\log (n+1)}}$ and then enumerate the
$N^{K\log (n+2)}\,k_{n+1}$-balls in $A_{n+1}$ denoted $B^{n+1}_1,\dots,B^{n+1}_{N^{K\log (n+2)}}$.
We first choose edges between $B^n_1$ and $B^{n+1}_1,B^{n+1}_2$ and assign flow $1/ N^{K\log
(n+2)}$ to the edge to $B^{n+1}_1$ and $1/ N^{K\log (n+1)}-1/ N^{K\log (n+2)}$ to the edge to
$B^{n+1}_2$ We then fill $B^{n+1}_2$ up to level $1/ N^{K\log (n+2)}$ from $B^n_2$ and successively
assign flows to edges from the $B^n_i$'s to the $B^{n+1}_j$'s so that each $B^n_i$ becomes empty
and each $B^{n+1}_j$ is filled up by the end of the procedure. This can be done in several ways so
that all the entrance flows and exit flows  have the same order of magnitude.

 This procedure is repeated for the
successive $A_n$. This means that for each $n>n_0$, each one of the $k_{n}$-balls in $A_n$ gets
$1/N^{K\log (n+1)}$ total amount of entrance flow. Noting that for large $n$ \[
\frac{1}{N^{K\log(n+1)}}<\frac{2}{N^{K\log(n+2)}},\] each of the $k_{n}$-balls has 1 or 2 entrance
edges and $1$, $2$ or $3$ exit edges.
 Any entrance-exit
pair in the $k_n$-balls can be connected by a path (within the ball by completeness) of expected  length
bounded by $CD(n_0) n$ (Lemma \ref{L.13}). Therefore each of the edges belongs to at most $6$ paths
and the expected energy of the flow (recall Definition \ref{D.energy})  is then bounded by
\begin{equation} 6\sum^\infty_{n=n_0}N^{K\log (n+1)}\frac{CD(n_0)
n}{N^{2K\log (n+1)}}= 6\sum^\infty_{n=n_0}\frac{CD(n_0) n}{(n+1)^{K\log N}},\end{equation}
 where the $n$th summand  refers to the expected energy  of the flow from entrance
vertices in $A_n$ to the entrance vertices in $A_{n+1}$. Then the expected energy of the flow is finite if
\[\sum_n\frac{n}{n^{K\log N}}<\infty,\]  which holds because $ K>2/\log N $. Hence with the assumption that $b>3K$ we can construct
a flow on a subgraph of $\mathcal{C}$  with finite energy for almost every realization of the percolation cluster and therefore the random walk on the
cluster is transient. Since this holds for $K\log N>2$ and  $b>3K$, and we have assumed that
$\alpha >b\log N$, then $\alpha
>6$ suffices.
\end{proof}
\bigskip

Finally, we can give the main result.

\begin{theorem}\label{T.main2}  Consider the simple random walk on the percolation cluster with
$\delta=1$,   $C_2>0$.  Then for almost every realization of the percolation cluster there exists a critical $\alpha_c\in (0,\infty)$
such that for $\alpha <\alpha_c$ the random walk is recurrent and for $\alpha >\alpha_c$ the random
walk is transient.

\end{theorem}

\begin{proof}  By Theorems  \ref{T.rec} and \ref{T.trans2}  there exist
$0<\alpha_1<\alpha_2<\infty$ such that the random walk on the percolation cluster is recurrent for
$\alpha\leq \alpha_1$ and transient for $\alpha=\alpha_2$.  Moreover, given $\alpha<\alpha'$ we can
construct the two associated percolation clusters (using the same $C_2$)  on one probability space
so that the $\alpha$-cluster is a subgraph of the $\alpha'$-cluster with probability one (see
Remark 2.2 in \cite{DGo2}). But then the Rayleigh monotonicity principle implies that if the random
walk on the $\alpha$-cluster is transient it is also transient on the $\alpha'$-cluster. We define
$\alpha_c=\inf\{\alpha:\text{the walk on the }\alpha\text{-cluster is transient}\}$, which yields
the desired result.
\end{proof}

\section{Further questions}

Questions that could be addressed which lie outside the scope of the present paper are as follows.

We have focussed on connection probabilities with $c_k$ having logarithmic and polynomial growth.
It would be interesting to study existence of percolation with intermediate growth, and related
questions for random walks.

What is the exact value of the critical $\alpha_c$, and is the walk recurrent or transient at
$\alpha=\alpha_c$?

\end{document}